\documentclass[a4paper,leqno,11pt]{article}
\usepackage{amsmath, amsthm, amssymb}
\usepackage{mathrsfs}
\usepackage{graphicx}

\newtheorem{thm}{Theorem}[section]

\newtheorem{lem}[thm]{Lemma}
\newtheorem{prop}[thm]{Proposition}

\theoremstyle{definition}

\makeatletter
\@addtoreset{equation}{section} 

\makeatother

\begin{document}

%
%
\newcommand{\BC}{\mathbb{C}}
\newcommand{\BR}{\mathbb{R}}
\newcommand{\BN}{\mathbb{N}}
\newcommand{\BZ}{\mathbb{Z}}

%
%
%

\title{A note on $G_q$-summability of formal solutions of some 
       linear $q$-difference-differential equations}

\author{{} \\
Hidetoshi \textsc{TAHARA}\footnote{Dept. of Information and 
        Communication Sciences, Sophia University, Tokyo 
        102-8554, Japan.
        e-mail: \texttt{h-tahara@sophia.ac.jp}}
~and Hiroshi \textsc{YAMAZAWA}\footnote{College of Engineering 
        and Design, Shibaura Institute of Technology, 
        Saitama 337-8570, Japan. 
         e-mail: \texttt{yamazawa@shibaura-it.ac.jp}}}

\date{}

\maketitle

%
%
%
%
\begin{abstract}
    Let $q>1$ and $\delta>0$. For a function $f(t,z)$, the $q$-shift 
operator $\sigma_q$ in $t$ is defined by $\sigma_q(f)(t,z)=f(qt,z)$.
This article discusses a linear $q$-difference-differential equation
$\sum_{j+\delta |\alpha| \leq m}
     a_{j,\alpha}(t,z)(\sigma_q)^j \partial_z^{\alpha}X
         =F(t,z)$
in the complex domain, and shows a result on the $G_q$-summability
of formal solutions (which may be divergent) in the framework of
$q$-Laplace and $q$-Borel transforms by Ramis-Zhang.
\end{abstract}

{\it Key words and phrases}: $q$-difference-differential equations,
       summability, formal power series solutions, 
                $q$-Gevrey asymptotic expansions.

{\it 2010 Mathematics Subject Classification Numbers}: 
       Primary 35C20; Secondary 35A01, 39A13.

\renewcommand{\thefootnote}{\fnsymbol{footnote}}

\footnote[0]{
      The first author is supported by JSPS KAKENHI Grant Number 
        JP15K04966.}

%
%
%
%

\section{Introduction}\label{section 1}

   Let $(t,z)$ be the variable in $\BC_t \times \BC_z^d$. Let $q>1$. 
For a function $f(t,z)$ we define a $q$-shift operator $\sigma_q$ 
in $t$ by $\sigma_q(f)(t,z)=f(qt,z)$.
\par
    In this note, we consider a linear $q$-difference-differential 
equation
\begin{equation}
   \sum_{j+\delta |\alpha| \leq m}
     a_{j,\alpha}(t,z)(\sigma_q)^j \partial_z^{\alpha}X
         =F(t,z)  \label{1.1}
\end{equation}
under the following assumptions:
\par
   (1) $q>1$, $\delta>0$ and $m \in \BN^*(=\{1,2,\ldots \})$;
\par
   (2) $a_{j,\alpha}(t,z)$ ($j+\delta|\alpha| \leq m$) and
$F(t,z)$ are holomorphic functions in a neighborhood of
$(0,0) \in \BC_t \times \BC_z^d$;
\par
   (3) (\ref{1.1}) has a formal power series solution
\begin{equation}
       X(t,z)= \sum_{n \geq 0} 
       X_n(z) t^n \in {\mathcal O}_R[[t]] \label{1.2}
\end{equation}
where ${\mathcal O}_R$ denotes the set of all holomorphic functions
on $D_R=\{z \in \BC^d \,;\, |z_i|<R \enskip (i=1,\ldots,d) \}$.
\par
\medskip
   Our basic problem is:
\par
\medskip
   {\bf Problem 1.1.}
    Under what 
condition can we get a true solution $W(t,z)$ of (\ref{1.1}) which
admits $\hat{X}(t,z)$ as a $q$-Gevrey asymptotic expansion of 
order 1 (in the sense of Definition 1.2 given below) ?

\par
\medskip
   For $\lambda \in \BC \setminus \{0\}$ and $\epsilon >0$ we set
\begin{align*}
    &{\mathscr Z}_{\lambda}
                   =\{- \lambda q^m \in \BC \,;\, m \in \BZ \}, \\
    &{\mathscr Z}_{\lambda, \epsilon}
    = \bigcup_{m \in \BZ} \{t \in \BC \setminus \{0\}\,;\, 
            |1+\lambda q^m/t| \leq \epsilon  \}.
\end{align*}
It is easy to see that if $\epsilon >0$ is sufficiently small 
the set ${\mathscr Z}_{\lambda, \epsilon}$ is a disjoint union 
of closed disks. For $r>0$ we write $D_r^*=\{t \in \BC \,;\,
0<|t|<r \}$. The following definition is due to 
Ramis-Zhang \cite{RZ}.

\par
\medskip
{\bf Definition 1.2.} 
    (1) Let $\hat{X}(t,z)
      =\sum_{n \geq 0}X_n(z)t^n \in {\mathcal O}_{R}[[t]]$
and let $W(t,z)$ be a holomorphic function on 
$(D_r^* \setminus {\mathscr Z}_{\lambda}) \times D_R$
for some $r>0$. We say that {\it $W(t,z)$ admits $\hat{X}(t,z)$ 
as a $q$-Gevrey asymptitoc expansion of order 1}, if there are 
$M>0$ and $H>0$ such that
\[
    \biggl| W(t,z)-\sum_{n=0}^{N-1} X_n(z)t^n \biggr|
      \leq \frac{M H^N}{\epsilon} q^{N(N-1)/2}|t|^N 
\]
holds on 
$(D_r^* \setminus {\mathscr Z}_{\lambda, \epsilon}) \times D_R$ 
for any $N=0,1,2,\ldots$ and any sufficiently small $\epsilon>0$.
\par
   (2) If there is a $W(t,z)$ as above, we say that the formal solution
$\hat{X}(t,z)$ is $G_q$-summable in the direction $\lambda$.

\par
\medskip
   A partial answer to Problem 1.1 was given in Tahara-Yamazawa 
\cite{yama}: in this paper, we will give an improvement of the 
result in \cite{yama}. As in \cite{yama}, we will use the framework of 
$q$-Laplace and $q$-Borel transforms via Jacobi theta function, 
developped by Ramis-Zhang \cite{RZ} and Zhang \cite{Z2}. 
\par
   Similar problems are discussed by Zhang \cite{Z1}, 
Marotte-Zhang \cite{MZ} and Ramis-Sauloy-Zhang \cite{RSZ} in the 
$q$-difference equations, and by Malek \cite{M1,M2}, 
Lastra-Malek \cite{LM} and Lastra-Malek-Sanz \cite{LMS} 
in the case of $q$-difference-differential equations. But,
their equations are different from ours.

%
%
%
%
%
\section{Main results}\label{section 2}

   For a holomorphic function $f(t,z)$ in a neighborhood of 
$(0,0) \in \BC_t \times \BC_z^d$, we define the order of the 
zeros of the function $f(t,z)$ at $t=0$ (we denote this by
$\mathrm{ord}_t(f)$) by
\[
    \mathrm{ord}_t(f) = \min \{ k \in \BN \,;\, 
               (\partial_t^kf)(0,z) \not\equiv 0 
               \enskip \mbox{near $z=0$} \}
\]
where $\BN=\{0,1,2,\ldots \}$.
\par
    For $(a,b) \in \BR^2$ we set $C(a,b)=\{(x,y) \in \BR^2\,;\,
x \leq a, y \geq b \}$. We define the $t$-Newton polygon
$N_t(\ref{1.1})$ of equation (\ref{1.1}) by
\[
    N_t(\ref{1.1})= \mbox{the convex hull of }
      \bigcup_{j+ \delta |\alpha| \leq m} 
               C(j, \mathrm{ord}_t(a_{j,\alpha})).
\]
\par
   In this note, we will consider the equation (\ref{1.1}) under the 
following conditions (A${}_1$) and (A${}_2$):
\par
\medskip
    (A${}_1$) There is an integer $m_0$ such that 
$0 \leq m_0<m$ and 
\[
        N_t(\ref{1.1})= \{(x,y) \in \BR^2 \,;\, x \leq m, \, 
                    y \geq \max\{0, x-m_0 \} \}.
\]
\par
   (A${}_2$) Moreover, we have
\[
    |\alpha|>0 \Longrightarrow 
     (j, \mathrm{ord}_t(a_{j,\alpha})) 
                     \in int(N_t(\ref{1.1})),
\]
where $int(N_t(\ref{1.1}))$ denotes the interior of the set
$N_t(\ref{1.1})$ in $\BR^2$. 
\par
   The figure of $N_t(\ref{1.1})$ is as in Figure \ref{Fig1}. 
In Figure \ref{Fig1}, the boundary of $N_t(\ref{1.1})$ consists 
of a horizontal half-line $\Gamma_0$, a segment $\Gamma_1$ and  
a vertical half-line $\Gamma_2$, and $k_i$ is the slope of 
$\Gamma_i$ for $i=0,1,2$.

\begin{figure}[htbp]
\begin{center}
%
\includegraphics[scale=0.6]{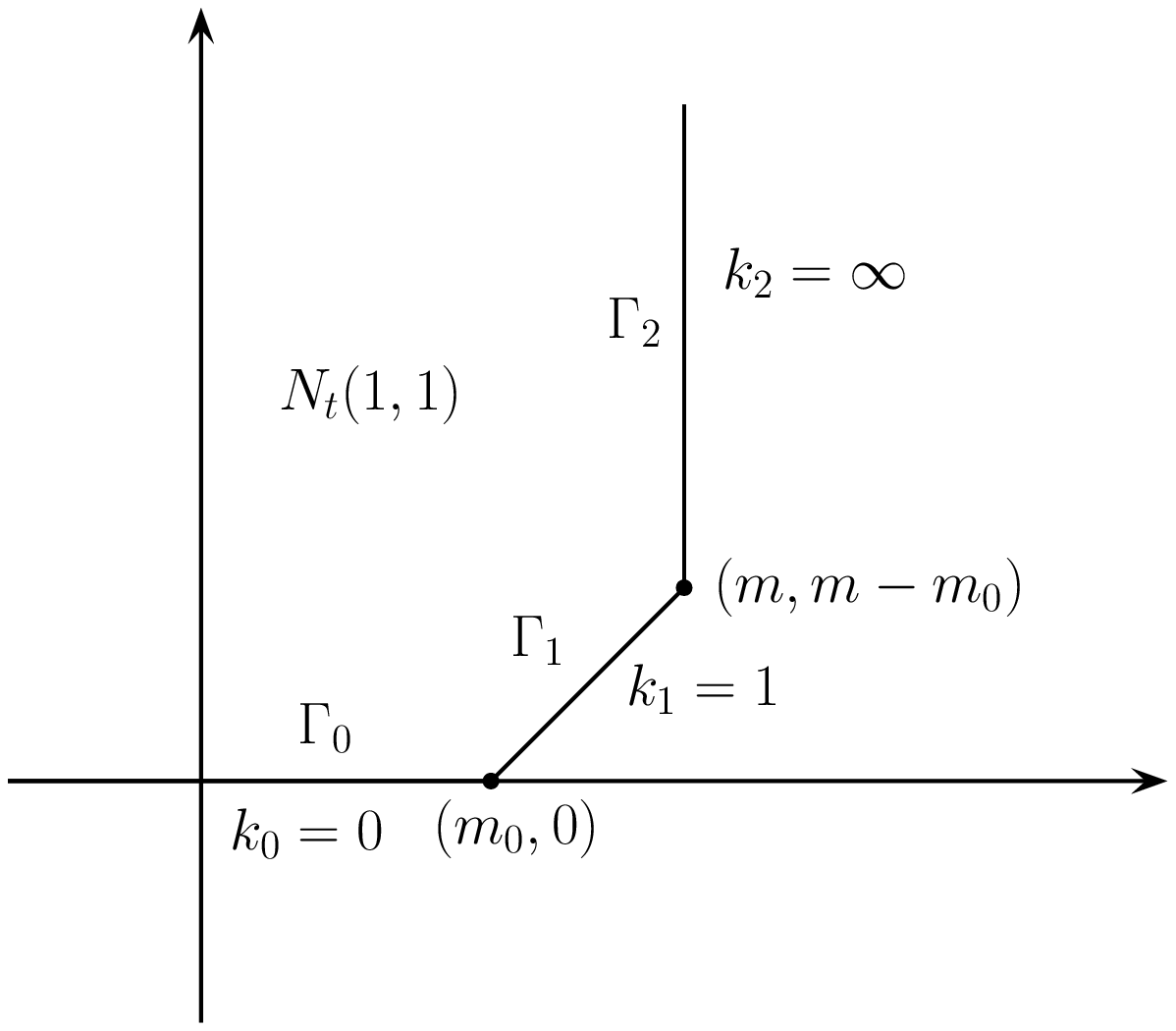}
%
%
\end{center}
\vspace{-3mm}
\caption{$t$-Newton polygon of $N_t(\ref{1.1})$} \label{Fig1}
\end{figure}

\begin{lem}\label{Lemma 2.1} 
    If {\rm (A${}_1$)} and {\rm (A${}_2$)} are satisfied, we have
\begin{equation}\label{2.1}
    \mathrm{ord}_t(a_{j,\alpha}) \geq 
    \left\{ \begin{array}{ll}
           \max \{0, j-m_0 \}, &\mbox{if $|\alpha|=0$},\\[2pt]
   \max \{1, j-m_0+1 \}, &\mbox{if $|\alpha|>0$}.
      \end{array}
          \right.
\end{equation}
\end{lem}

   By the condition (\ref{2.1}), we have the expression
\begin{equation}\label{2.2}
    a_{j,0}(t,z)=t^{j-m_0} b_{j,0}(t,z) \quad 
           \mbox{for $m_0<j \leq m$} 
\end{equation}
for some holomorphic functions $b_{j,0}(t,z)$ 
($m_0<j \leq m$) in a neighborhood of $(0,0) \in \BC \times \BC_z^d$.
We suppose:
\begin{equation}\label{2.3}
    a_{m_0,0}(0,0) \ne 0 \quad \mbox{and} \quad
    b_{m,0}(0,0) \ne 0.
\end{equation}
We set
\begin{equation}\label{2.4}
     P(\tau,z)= \sum_{m_0<j \leq m}
            \frac{b_{j,0}(0,z)}{q^{j(j-1)/2}} \tau^{j-m_0}
              + \frac{a_{m_0,0}(0,z)}{q^{m_0(m_0-1)/2}}
\end{equation}
and denote by $\tau_1, \ldots, \tau_{m-m_0}$ the
roots of $P(\tau,0)=0$.  By (\ref{2.3}) we have $\tau_i \ne 0$ 
for all $i=1,2,\ldots,m-m_0$. The set $S$ of singular directions 
at $z=0$ is defined by 
\[
     S = \bigcup_{i=1}^{m-m_0}
       \{t= \tau_i \eta \,;\, \eta >0 \}.
\]
\par
    In \cite{yama}, we have shown the following result.

\begin{thm}[Theorem 2.3 in \cite{yama}]\label{Theorem 2.2}
    {\rm (1)} Suppose the conditions {\rm (A${}_1$)}, 
{\rm (A${}_2$)} and {\rm (\ref{2.3})}.
Then, if equation {\rm (\ref{1.1})} has a formal solution
$\hat{X}(t,z)= \sum_{n \geq 0}X_n(z)t^n \in {\mathcal O}_R[[t]]$,
we can find $A>0$, $h>0$ and $0<R_1<R$ such that
$|X_n(z)| \leq A h^n q^{n(n-1)/2}$ on $D_{R_1}$ for any
$n=0,1,2,\ldots$.
\par
    {\rm (2)} In addition, if the condition
\begin{equation}\label{2.5}
    \mathrm{ord}_t(a_{j,\alpha}) \geq j-m_0+2, \quad
       \mbox{if $|\alpha|>0$ and $m_0 \leq j<m$}
\end{equation}
is satisfied, for any $\lambda \in \BC \setminus(\{0\} \cup S)$ 
the formal solution $\hat{X}(t,z)$ is $G_q$-summable in the 
direction $\lambda$. In other words, there are $r>0$, $R_1>0$ and 
a holomorphic solution $W(t,z)$ of (\ref{1.1}) on 
$(D_r^* \setminus {\mathscr Z}_{\lambda}) \times D_{R_1}$
such that $W(t,z)$ admits $\hat{X}(t,z)$ as a 
$q$-Gevrey asymptitoc expansion of order 1.
\end{thm}

   In this paper, we remove the additional condition
(\ref{2.5}) from the part (2) of Theorem \ref{Theorem 2.2}. We have

\begin{thm}\label{Theorem 2.3}
    Suppose the conditions {\rm (A${}_1$)}, {\rm (A${}_2$)} and 
{\rm (\ref{2.3})}.  Then, for any 
$\lambda \in \BC \setminus(\{0\} \cup S)$ the formal solution 
$\hat{X}(t,z)$ {\rm (}in {\rm (\ref{1.2})}{\rm )} is $G_q$-summable 
in the direction $\lambda$.
\end{thm}

   To prove this, we use the framework of $q$-Laplace and
$q$-Borel transforms developped by Rramis-Zhang \cite{RZ}.
By (1) of Theorem \ref{Theorem 2.2} we know that the formal 
$q$-Borel transform of $\hat{X}(t,z)$ in $t$
\begin{equation}\label{2.6}
    u(\xi,z)= \sum_{k \geq 0} \frac{X_k(z)}{q^{k(k-1)/2}}\xi^k
\end{equation}
is convergent in a neighborhood of $(0,0) \in \BC_{\xi} \times \BC_z^n$.
For $\lambda \in \BC \setminus \{0\}$ and $\theta>0$ we write
$S_{\theta}(\lambda)=\{\xi \in \BC \setminus \{0\} \,;\, 
|\arg \xi- \arg \lambda| <\theta \}$.  Then, to show 
Theorem \ref{Theorem 2.3} it is enough to prove the following 
result.

\par
\begin{prop}\label{Proposition 2.4}
    For any $\lambda \in \BC \setminus (\{0\} \cup S)$ there are 
$\theta>0$, $R_1>0$, $C>0$ and $H>0$ such that $u(\xi,z)$ has an 
analytic extension $u^*(\xi,z)$ to the domain 
$S_{\theta}(\lambda) \times D_{R_1}$ satisfying the following 
condition:
\begin{equation}\label{2.7}
   |u^*(\lambda q^m,z)| \leq C H^m q^{m^2/2}
    \quad \mbox{on $D_{R_1}$}, \quad m=0,1,2,\ldots.
\end{equation}
\end{prop}

%
%
%
\section{Some lemmas}\label{section 3}

   Before the proof of Proposition \ref{Proposition 2.4}, 
let us give some lemmas which are needed in the proof of 
Proposition \ref{Proposition 2.4}.
\par
   The following is the key lemma of the proof of 
Proposition \ref{Proposition 2.4}.

\begin{lem}\label{Lemma 3.1} 
   Let $q>1$.  Let $f(t,z)$ be a function in $(t,z)$. 
\par
   {\rm (1)} We have 
$\sigma_q(f)(t,z)=(\sigma_{\sqrt{q}})^2(f)(t,z)$.
\par
   {\rm (2)} We set $F(t,z)=f(t^2,z)$: then
we have  $\sigma_q(f)(t^2,z)=\sigma_{\sqrt{q}}(F)(t,z)$.
Similarly, we have
$(\sigma_q)^m(f)(t^2,z)=(\sigma_{\sqrt{q}})^m(F)(t,z)$
for any $m=1,2,\ldots$.
\end{lem}

\begin{proof}
    (1) is clear. (2) is verified as follows:
$\sigma_q(f)(t^2,z)=f(qt^2,z)= f((\sqrt{q}t)^2,z)$
      $= F(\sqrt{q}t,z) = \sigma_{\sqrt{q}}(F)(t,z)$. 
The equality 
$(\sigma_q)^m(f)(t^2,z)=(\sigma_{\sqrt{q}})^m(F)(t,z)$
can be proved in the same way.
\end{proof}

   The following result is proved in [Proposition 2.1 in 
\cite{ramis}]:

\begin{prop}\label{Proposition 3.2}
    Let $\hat{f}(t)= \sum_{n \geq 0}a_nt^n
\in \BC[[t]]$. The following two conditions are equivalent:
\par
   {\rm (1)} There are $A>0$ and $H>0$ such that
\[
         |a_n| \leq \dfrac{A H^n}{q^{n(n-1)/2}}, 
         \quad n=0,1,2,\ldots.  
\]
\par
   {\rm (2)} $\hat{f}(t)$ is the Taylor expansion at $t=0$ of
an entire function $f(t)$ satisfying the estimate
\[
     |f(t)| \leq M 
        \exp \Bigl( \dfrac{(\log |t|)^2}{2 \log q} 
                    + \alpha \log|t| \Bigr)
     \quad \mbox{on $\BC \setminus \{0\}$} 
\]
for some $M>0$ and $\alpha \in \BR$.
\end{prop}

%
%
%
\section{Proof of Proposition \ref{Proposition 2.4}}
\label{section 4}

    We set $q_1=q^{1/4}$, replace $t$ by $t^2$ in (\ref{1.1}), 
and apply Lemma \ref{Lemma 3.1} to 
the equation (\ref{1.1}): then (\ref{1.1}) is rewritten into 
the form
\begin{equation}\label{4.1}
    \sum_{j+\delta |\alpha| \leq m}A_{j,\alpha}(t,z)
           (\sigma_{q_1})^{2j} \partial_z^{\alpha}Y
           = G(t,z) 
\end{equation}
where 
\begin{align*}
    &A_{j,\alpha}(t,z)= a_{j,\alpha}(t^2,z) \quad
            (j+\delta |\alpha| \leq m), \\
    &Y(t,z)=X(t^2,z) = \sum_{k \geq 0} X_k(z)t^{2k}, \\
    &G(t,z)= F(t^2,z).
\end{align*}
We can regards (\ref{4.1}) as a $q_1$-difference-differential 
equation, and in this case, the order of the equation is $2m$ 
in $t$.  Therefore, the $t$-Newton polygon $N_t(\ref{4.1})$ of 
(\ref{4.1}) (as a $q_1$-difference equation) is 
\[
        N_t(\ref{4.1})= \{(x,y) \in \BR^2 \,;\, x \leq 2m, \, 
                    y \geq \max\{0, x-2m_0 \} \}
\]
which is as in Figure \ref{Fig2}.

\begin{figure}[htbp]
\begin{center}
%
\includegraphics[scale=0.6]{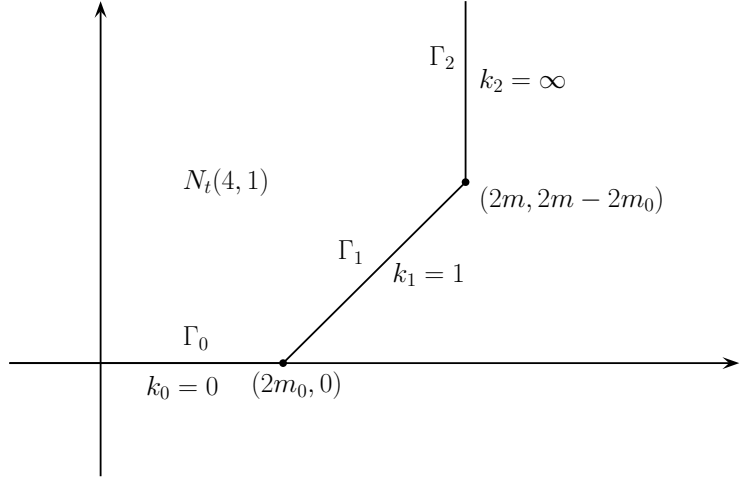}
%
%
\end{center}
\vspace{-3mm}
\caption{$t$-Newton polygon of (4.1)}\label{Fig2}
\end{figure}

Moreover, we have
\begin{equation}\label{4.2}
    \mathrm{ord}_t(A_{j,\alpha}) \geq 
    \left\{ \begin{array}{ll}
           \max \{0, 2j-2m_0 \}, &\mbox{if $|\alpha|=0$},\\[2pt]
           \max \{2, 2j-2m_0+2 \}, &\mbox{if $|\alpha|>0$}.
      \end{array}  \right. 
\end{equation}
By (\ref{2.2}) we have 
\[
    A_{j,0}(t,z)=t^{2j-2m_0} B_{j,0}(t,z) \quad 
           \mbox{for $m_0<j \leq m$}
\]
for $B_{j,0}(t,z)=b_{j,0}(t^2,z)$ ($m_0<j \leq m$). The set 
$S_1$ of singular directions of (\ref{4.1}) is defined by using
\begin{align*}
    P_1(\rho,z) &=\sum_{m_0<j \leq m}
            \frac{B_{j,0}(0,z)}{{q_1}^{2j(2j-1)/2}} \rho^{2j-2m_0}
             + \frac{A_{m_0,0}(0,z)}{{q_1}^{2m_0(2m_0-1)/2}} \\
   &= \sum_{m_0 <j \leq m}
            \frac{b_{j,0}(0,z)}{{q_1}^{2j(2j-1)/2}} \rho^{2j-2m_0}.
         + \frac{a_{m_0,0}(0,z)}{{q_1}^{2m_0(2m_0-1)/2}}.
\end{align*}
Let $\rho_1, \ldots, \rho_{2m-2m_0}$ be the roots of $P_1(\rho,0)=0$:
then $S_1$ is defined by
\[
     S_1 = \bigcup_{i=1}^{2m-2m_0}
       \{t= \rho_i \eta \,;\, \eta >0 \}.
\]
Let $u_1(\xi,x)$ be the $q_1$-formal Borel transform of $Y(t,x)$, 
that is, 
\[
     u_1(\xi,z)= \sum_{k \geq 0}
           \frac{X_k(z)}{{q_1}^{2k(2k-1)/2}} \xi^{2k}.
\]
Since $q_1=q^{1/4}$ we can easily see:
\begin{align}
    &u_1(\xi,z) = u(q^{-1/4}\xi^2,z), \label{4.3} \\
    &P_1(\lambda,z) = q^{-m_0/4} P(q^{-1/4} \lambda^2,z),
                    \label{4.4}
\end{align}
where $u(\xi,z)$ and $P(\tau,z)$ are the ones in (\ref{2.6}) and 
(\ref{2.4}), respectively.
\par
    By (\ref{4.3}) we see that $u_1(\xi,z)$ is convergent in a 
neighborhood of $(\xi,z)=(0,0)$. The equality (\ref{4.4}) implies 
that $\lambda \in \BC \setminus (\{0\} \cup S_1)$ is equivalent
to the condition $\lambda^2 \in \BC \setminus (\{0\} \cup S)$.
\par
   Since $\textrm{ord}_t(A_{j,\alpha}) \geq 2j-2m_0+2$ holds for
any $(j,\alpha)$ with $m_0 \leq j<m$ and $|\alpha|>0$, 
the $q_1$-difference equation (\ref{4.1}) satisfies the condition 
(\ref{2.5}) (with $j$, $m_0$, $m$ replaced by $2j$, $2m_0$, $2m$, 
respectively). Therefore, we can apply (2) of Theorem \ref{Theorem 2.2} 
and its proof to the equation (\ref{4.1}). 
\par
    In particular, by the proof of [Proposition 5.6 in \cite{yama}] 
we have

\begin{prop}\label{Proposition 4.1} 
   For any $\rho \in \BC \setminus (\{0\} \cup S_1)$ we can find
$\theta_1>0$ and $R_1>0$ which satisfy the following 
conditions {\rm (1)} and {\rm (2)}:
\par
   {\rm (1)} $u_1(\xi,z)$ has an analytic extension $u^*_1(\xi,z)$
to the domain $S_{\theta_1}(\rho) \times D_{R_1}$.
\par
   {\rm (2)} There are $\mu>0$ and holomorphic functions $w_n(\xi,z)$
{\rm (}$n \geq \mu${\rm )} on $S_{\theta_1}(\rho) \times D_{R_1}$ which satisfy
\begin{equation}\label{4.5}
      u^*_1(\xi,z)= \sum_{n \geq 2\mu}w_n(\xi,z)+
                 \sum_{0 \leq k <\mu}
               \frac{X_k(z)}{{q_1}^{2k(2k-1)}}\xi^{2k}
      \quad \mbox{on $S_{\theta_1}(\rho) \times D_{R_1}$}
\end{equation}
and
\[
     |w_n(\xi,z)| \leq \frac{AH^n |\xi|^n}{{q_1}^{n(n-1)/2}}
     \quad \mbox{on $S_{\theta_1}(\rho) \times D_{R_1}$},
    \quad n \geq 2\mu  
\]
for some $A>0$ and $H>0$.
\end{prop}

   Therefore, by applying Proposition \ref{Proposition 3.2} to 
(\ref{4.5}) we have the estimate
\begin{equation}\label{4.6}
   |u_1^*(\xi,x)| \leq M 
       \exp \Bigl( \frac{(\log |\xi|)^2}{2 \log q_1}
                + \alpha \log|\xi| \Bigr)
   \quad \mbox{on $S_{\theta_1}(\rho)
              \times D_{R_1}$}
\end{equation}
for some $M>0$ and $\alpha \in \BR$.

\begin{proof}[Completion of the proof of 
                        Proposition \ref{Proposition 2.4}]
    Take any $\lambda=r e^{\sqrt{-1} \theta} \in 
              \BC \setminus (\{0\} \cup S)$. We set
$\rho= \sqrt{r} e^{\sqrt{-1} \theta/2}$: then
we have $\rho \in \BC \setminus (\{0\} \cup S_1)$.
Therefore, by Proposition \ref{Proposition 4.1} we can get 
$\theta_1>0$, $R_1>0$, $M>0$ and $\alpha \in \BR$ such that 
$u_1(\xi,z)$ has an analytic extension $u_1^*(\xi,z)$ 
to the domain $S_{\theta_1}(\rho) \times D_{R_1}$ satisfying 
the estimate (\ref{4.6}) on $S_{\theta_1}(\rho) \times D_{R_1}$.
\par
   Since $u_1(\xi,z)= u(q^{-1/4} \xi^2,z)$ holds, this shows that
$u(\xi,z)$ has also an analytic continuation $u^*(\xi,x)$ to the domain
$S_{\theta}(\lambda) \times D_{R_1}$ (with $\theta=2 \theta_1$), and
we have $u^*(\xi,z)=u_1^*(q^{1/8} \xi^{1/2},z)$ on 
$S_{\theta}(\lambda) \times D_{R_1}$. Therefore, by (\ref{4.6}) 
we have the estimate
\begin{align*}
   |u^*(\xi,x)| &\leq M \exp \Bigl( 
       \frac{(\log (q^{1/8}|\xi|^{1/2}))^2}{2 \log q^{1/4}}
         + \alpha \log (q^{1/8}|\xi|^{1/2}) \Bigr) \\
     &= M_1 |\xi|^{\beta} \exp 
        \Bigl( \frac{(\log |\xi|)^2}{2 \log q} \Bigr) 
     \quad \mbox{on $S_{\theta}(\lambda) \times D_{R_1}$}
\end{align*}
(with $M=M_1 q^{1/32+\alpha/8}$ and $\beta=1/4+\alpha/2$).
\par
   Thus, by setting $\xi=\lambda q^m$ we obtain
\begin{align*}
   |u^*(\lambda q^m,x)| &\leq M_1 |\lambda q^m|^{\beta} \exp 
        \Bigl( \frac{(\log |\lambda q^m|)^2}{2 \log q} \Bigr) \\
    &= M_1 |\lambda|^{\beta}
       \exp \Bigl( \frac{(\log |\lambda|)^2}{2 \log q}
          \Bigr) (|\lambda| q^{\beta})^m q^{m^2/2},
         \quad m=0,1,2,\ldots.
\end{align*}
This proves (2.7). 
\end{proof}

%
%

%
%
\end{document}